\documentclass[11pt,a4paper]{article}
\usepackage[T1]{fontenc}
\usepackage{lmodern,amsmath,amssymb,amsthm,booktabs,array,microtype}
\usepackage[margin=25mm]{geometry}
\usepackage[hidelinks]{hyperref}
\newtheorem{theorem}{Theorem}[section]
\newtheorem{proposition}[theorem]{Proposition}
\newtheorem{corollary}[theorem]{Corollary}
\newtheorem{lemma}[theorem]{Lemma}
\theoremstyle{definition}\newtheorem{definition}[theorem]{Definition}
\theoremstyle{remark}\newtheorem{remark}[theorem]{Remark}
\newcommand{\kk}{\mathbb F}
\newcommand{\Hilb}{\operatorname{Hilb}}
\newcommand{\Soc}{\operatorname{Soc}}
\newcommand{\BC}{\operatorname{BC}}

\newcommand{\gr}{\operatorname{gr}}
\title{The Negami Polynomial and Broken-Circuit\\Stanley--Reisner Rings\\[5pt]\large Classical Hilbert Identities and a Rank-Filtered Algebra\\of Spanning Subgraphs}
\author{IWAO MIZUKAI\\\small Chiba Keizai University High School, Chiba, Japan\\\small\texttt{get\_mizukai@hotmail.com}}
\date{Revised October 11, 2026}
\begin{document}
\maketitle
\begin{abstract}
We summarize the classical passage from the chromatic specialization of Negami's polynomial to the Hilbert series of a broken-circuit Stanley--Reisner ring, and record formulas for several graph families with their connectivity hypotheses made explicit. Our main construction in this revision is the associated graded algebra of the squarefree edge algebra under the graphic-matroid rank filtration, for an arbitrary finite loopless graph. Its basis is indexed by all edge subsets, and its multiplication is governed by disjointness and rank additivity. Its bigraded Hilbert polynomial recovers the full three-variable Negami polynomial. We prove that its socle has a basis indexed by edge subsets of full matroid rank; for a connected graph these are precisely the connected spanning subgraphs. A vertex-degree refinement counts connected regular spanning subgraphs. This finite-dimensional algebra is distinct from a broken-circuit Stanley--Reisner ring. We distinguish these elementary structural results from the unresolved question of whether abstract algebra isomorphism can distinguish graphs with the same Negami polynomial; no priority claim is made for the construction or its Hilbert identity.
\end{abstract}

\section{Introduction}
Negami's polynomial counts spanning subgraphs by edges and connected components \cite{Negami}. Its relation to Tutte invariants is established \cite{OxleyNegami,State}. The connection between chromatic coefficients, no-broken-circuit sets and Stanley--Reisner Hilbert series is likewise classical \cite{Whitney,BO,BZ,Stanley}. Sections 2--8 collect this background and its applications concisely; they are not offered as new Hilbert-series identities.

Section 9 develops the rank-filtered construction beyond the triangle example. We give its multiplication, prove a socle formula including the absence of cancellation in linear combinations, and explain a vertex-degree refinement. Section 10 identifies what information is already determined by the polynomial and what would be needed for a stronger distinction result. The original section numbering is retained to make the corrected hypotheses and the expanded Section 9 easy to locate.

\section{Prior work and scope}
Whitney's theorem and the usual face-to-$h$ transformation imply the correspondence in Section 5. Proudfoot--Speyer construct a different broken circuit ring with a Gr\"obner degeneration to a broken-circuit Stanley--Reisner ring \cite{PS}. Berget realizes Tutte evaluations using spaces of products of linear forms \cite{Berget}. These are relevant precedents for algebraic realizations of graph-polynomial data.

Llamas--Mart\'inez-Bernal--Merino \cite{LMM} treat structural properties as well as Hilbert identities. Their Theorem 9 gives an edge order for the join of a simple graph with $K_n$, $n\geq1$, such that the reduced broken-circuit complex is doubly pure shellable and doubly Cohen--Macaulay, and the associated Stanley--Reisner rings are level. Their Example 11 shows that the broken-circuit ring of $K_{2,3}$ over $\mathbb Q$ is not level. Thus an additional Hilbert substitution alone does not go beyond the scope of that work.

The algebra in Section 9 instead uses \emph{all} edge subsets, is finite-dimensional, and requires no edge order. Its socle is computed directly from multiplication. This difference of objects does not itself establish novelty. We use ``Negami--Hilbert correspondence'' only as a local name for a composition of known maps, and make no assertion that the rank-filtered construction or the socle description is absent from the literature.

\section{The Negami polynomial and chromatic specialization}
Graphs are finite. Parallel edges are allowed unless simplicity is stated; Section 9 excludes loops. Write $n=|V|$, $m=|E|$, $c=k(E)$, and
\[
 k(A)=\#\text{components of }(V,A),\qquad \rho(A)=n-k(A),\qquad r=\rho(E)=n-c.
\]
Isolated vertices count as components. We use the spanning-subgraph normalization
\begin{equation}\label{state}
 f(G;t,x,y)=\sum_{A\subseteq E}t^{k(A)}x^{|A|}y^{m-|A|}.
\end{equation}
It gives $f(\overline K_n;t,x,y)=t^n$, deletion--contraction $f(G)=x f(G/e)+y f(G-e)$ for a non-loop edge, and $f(G)=(x+y)f(G-e)$ for a loop. In particular,
\[
 f(G;q,-1,1)=\sum_{A\subseteq E}(-1)^{|A|}q^{k(A)}=P_G(q).
\]
The last equality is the chromatic inclusion--exclusion formula.

\section{Broken-circuit Stanley--Reisner rings}
Fix a total order on the edges of a loopless graph. A broken circuit is $C\setminus\{\max C\}$ for a circuit $C$ of the graphic matroid; in particular a parallel pair is a circuit of size two. The NBC sets form $\BC(G)$. Over a field $\kk$,
\[
 \kk[\BC(G)]=\kk[x_e:e\in E]\Big/\left(\prod_{e\in C\setminus\{\max C\}}x_e:C\text{ a circuit}\right).
\]
Its dimension is $r=n-c$. If $b_i$ counts NBC sets of size $i$, set
\[
 h_G(z)=\sum_{i=0}^r b_i z^i(1-z)^{r-i},\qquad
 \Hilb(\kk[\BC(G)];z)=\frac{h_G(z)}{(1-z)^r}.
\]
The largest-edge convention is equivalent, by reversing the order, to the least-edge convention of \cite{LMM}.

\section{The classical Negami--Hilbert correspondence}
\begin{theorem}[Chromatic/NBC form]\label{classical}
Let $G$ be finite, connected and loopless, and let $r=n-1$. Then
\[
 \boxed{h_G(z)=(-z)^r\left[\frac{f(G;q,-1,1)}q\right]_{q=(z-1)/z}.}
\]
\end{theorem}
\begin{proof}
Whitney's theorem gives $\chi_{M(G)}(q)=P_G(q)/q=\sum_i(-1)^i b_iq^{r-i}$. Substitution yields $(-z)^r\chi_{M(G)}((z-1)/z)=\sum_i b_i z^i(1-z)^{r-i}$.
\end{proof}
\begin{remark}
For $c$ components, replace $q$ in the denominator by $q^c$ and use $r=n-c$. The identity uses only the chromatic slice of $f$, not its full coefficient data.
\end{remark}

\section{Infinite graph families}
All formulas below are consequences of Theorem~\ref{classical}. The Hilbert series in every case is $h_G(z)/(1-z)^r$.
\begin{theorem}[Cycles]
For $n\geq3$, $f(C_n;t,x,y)=(x+ty)^n+(t-1)x^n$ and
\[
 I_{\BC(C_n)}=(x_1\cdots x_{n-1}),\qquad h_{C_n}(z)=1+z+\cdots+z^{n-2}.
\]
\end{theorem}
\begin{proof}
There is one circuit. Alternatively use $P_{C_n}(q)=(q-1)^n+(-1)^n(q-1)$.
\end{proof}
\begin{theorem}[Wheels]
For $W_n=K_1+C_n$ and $n\geq3$, $r=n$ and
\[
 h_{W_n}(z)=(1+z)^n-z^{n-1}-z^n.
\]
\end{theorem}
\begin{proof}
Color the hub first: $P_{W_n}(q)=q((q-2)^n+(-1)^n(q-2))$.
\end{proof}
\begin{theorem}[Complete graphs]
For $n\geq2$, $r=n-1$ and
\[
 h_{K_n}(z)=\prod_{j=1}^{n-2}(1+jz),\qquad
 h_i=\left[{n-1\atop n-1-i}\right]\quad(0\leq i\leq n-2),
\]
where brackets denote unsigned Stirling numbers of the first kind.
\end{theorem}
\begin{proof}
Substitute $P_{K_n}(q)/q=\prod_{j=1}^{n-1}(q-j)$ in Theorem~\ref{classical}.
\end{proof}
\begin{theorem}[Complete bipartite graphs]
For $a,b\geq1$, $r=a+b-1$, and braces denoting Stirling numbers of the second kind,
\[
 h_{K_{a,b}}(z)=\sum_{j=1}^a(-1)^{a+j}\left\{{a\atop j}\right\}z^{a-j}
 \left(\prod_{\ell=1}^{j-2}(1+\ell z)\right)(1+(j-1)z)^b.
\]
Empty products are one. In particular $h_{K_{2,b}}=(1+z)^b-z$ and
$h_{K_{3,b}}=z^2-3z(1+z)^b+(1+z)(1+2z)^b$.
\end{theorem}
\begin{proof}
Counting colors on the first part gives $P_{K_{a,b}}(q)=\sum_{j=1}^a\left\{{a\atop j}\right\}(q)_j(q-j)^b$. Divide by $q$ and substitute as above.
\end{proof}

\section{Strongly regular graphs}
\begin{theorem}[Corrected connectivity hypothesis]
Let $G=\operatorname{srg}(v,k,\lambda,\mu)$ be a \textbf{connected}, nontrivial, simple strongly regular graph. Put $m=vk/2$, $r=v-1$ and $\beta=m-r$. Then
\[
 h_0=1,\qquad h_1=\beta,\qquad h_2=\binom{\beta+1}{2}-\frac{vk\lambda}{6}.
\]
Moreover $\dim\kk[\BC(G)]=v-1$ and its codimension in $\kk[x_e:e\in E]$ is $\beta$.
\end{theorem}
\begin{proof}
For any simple graph, $b_0=1$, $b_1=m$, and $b_2=\binom m2-\tau$, where $\tau$ is its number of triangles. Distinct triangles give distinct two-edge broken circuits. Here $\tau=m\lambda/3=vk\lambda/6$. The face-to-$h$ transformation gives
\[
 h_1=m-r,\qquad h_2=b_2-(r-1)m+\binom r2
 =\binom{m-r+1}{2}-\tau.
\]
Connectivity gives $r=v-1$.
\end{proof}
\begin{remark}[Why connectivity matters]
For $C_3\sqcup C_3$, the rank is $6-2=4$ and $h_G(z)=(1+z)^2$. The erroneous choice $r=5$ would give $h_1=1$ and $h_2=-1$. A disconnected version is valid with $r=v-c$ and $\beta=m-v+c$ throughout.
\end{remark}

\section{Girth and regular families}
\begin{theorem}[Initial $h$-coefficients]
Let $G$ be simple and connected with finite girth $g$, and let $\beta=m-n+1$. Then
\[
 h_j=\binom{\beta+j-1}{j}\quad(0\leq j\leq g-2),\qquad
 h_{g-1}=\binom{\beta+g-2}{g-1}-c_g(G),
\]
where $c_g(G)$ counts simple cycles of length $g$.
\end{theorem}
\begin{proof}
For $i\leq g-2$, $b_i=\binom mi$. A broken circuit of size $g-1$ is the path left by deleting the largest edge of a shortest cycle. In a simple graph that path determines its closing edge, so distinct shortest cycles yield distinct broken circuits. Hence $b_{g-1}=\binom m{g-1}-c_g(G)$. Substitute in
\[
 h_j=\sum_{i=0}^j(-1)^{j-i}\binom{r-i}{j-i}b_i.
\]
The initial coefficients agree with those of $(1-z)^{r-m}=(1-z)^{-\beta}$ until the indicated subtraction.
\end{proof}
\begin{corollary}[Corrected connectivity hypothesis]
Let $G_N$ be a family of \textbf{simple connected} $d$-regular graphs on $N$ vertices, with fixed $d\geq3$ and $N\to\infty$. For fixed $j\geq0$, assume eventually $g(G_N)\geq j+2$. Then
\[
 \beta_N=\left(\frac d2-1\right)N+1,\qquad
 h_j(G_N)=\binom{\beta_N+j-1}{j}
 \sim\frac1{j!}\left(\frac{d-2}{2}N\right)^j.
\]
\end{corollary}
This applies to Ramanujan families satisfying these hypotheses; the Ramanujan spectral condition alone does not assert increasing girth. The coefficient formula is a girth consequence and does not use the spectral bound.

\section{The rank-filtered algebra and its socle}
Throughout this section $G$ is a finite loopless graph, possibly disconnected and with parallel edges, and $\kk$ is any field. Define
\[
 B_G=\kk[z_e:e\in E]/(z_e^2:e\in E),\qquad z_A=\prod_{e\in A}z_e.
\]
The $2^m$ monomials $z_A$ form a vector-space basis. This is a commutative square-zero-variable algebra, not an exterior algebra.

\begin{definition}[Rank filtration]
Set $F_{-1}B_G=0$ and $F_pB_G=\operatorname{span}_{\kk}\{z_A:\rho(A)\leq p\}$ for $p\geq0$. Write
\[
 N(G)=\gr_F B_G=\bigoplus_{p\geq0}F_pB_G/F_{p-1}B_G,
 \qquad u_A=z_A+F_{\rho(A)-1}B_G.
\]
\end{definition}
\begin{lemma}[Multiplication]\label{mult}
The filtration is multiplicative. The elements $u_A$, $A\subseteq E$, form a basis of $N(G)$, with
\begin{equation}\label{product}
 u_Au_B=\begin{cases}
 u_{A\cup B},&A\cap B=\varnothing\ \text{and }\rho(A\cup B)=\rho(A)+\rho(B),\\
 0,&\text{otherwise}.
 \end{cases}
\end{equation}
The unit is $u_\varnothing$, and $\deg u_A=(\rho(A),|A|)$ is a bigrading.
\end{lemma}
\begin{proof}
If $A\cap B\neq\varnothing$ the product in $B_G$ is zero. Otherwise it is $z_{A\cup B}$ and rank subadditivity places it in $F_{\rho(A)+\rho(B)}$. It survives in that graded quotient exactly when equality holds. Each basis monomial enters the filtration at its own rank; this also proves the asserted basis. Associativity is inherited from the associated graded construction.
\end{proof}
\begin{proposition}[Full Hilbert realization]\label{hilbertfull}
For every such $G$,
\[
 H_G(u,s):=\Hilb(N(G);u,s)=\sum_{A\subseteq E}u^{\rho(A)}s^{|A|},\qquad
 \boxed{f(G;t,x,y)=t^n y^m H_G(t^{-1},x/y).}
\]
The latter is a polynomial identity after cancellation of the displayed negative powers.
\end{proposition}
\begin{proof}
Apply the basis in Lemma~\ref{mult} and $n-\rho(A)=k(A)$. Each substituted term is $t^{k(A)}x^{|A|}y^{m-|A|}$.
\end{proof}
The known values $n,m$ are part of this recovery. The abstract algebra does not record isolated vertices, whereas $f(G)$ acquires a factor $t$ for each added isolated vertex.

\subsection*{The triangle, with its extra algebra generator}
For $C_3$, the original computation is retained:
\[
 H_{C_3}(u,s)=1+3us+3u^2s^2+u^2s^3,
\]
\[
 f(C_3;t,x,y)=t^3y^3+3t^2xy^2+3tx^2y+tx^3.
\]
For distinct edges $1,2,3$, $u_{\{1\}}u_{\{2\}}=u_{\{1,2\}}\neq0$, but
\[
 u_{\{1,2\}}u_{\{3\}}=0,\qquad u_{\{1,2,3\}}\neq0.
\]
In particular $u_{\{1,2,3\}}$ is \emph{not} the product of the three singleton classes. The algebra is not in general generated by its edge-degree-one part. This distinction is essential when discussing generators, the socle, or standard graded levelness.

\begin{theorem}[Socle basis]\label{socle}
Let $r=\rho(E)$. The algebra $N(G)$ is a finite-dimensional local algebra with maximal ideal
\[
 \mathfrak m=\operatorname{span}_{\kk}\{u_A:A\neq\varnothing\}.
\]
Its socle is
\begin{equation}\label{soc}
 \boxed{\Soc N(G)=(0:\mathfrak m)
 =\bigoplus_{\rho(A)=r}\kk u_A.}
\end{equation}
If $G$ is connected, the indexing sets are exactly the connected spanning subgraphs $(V,A)$.
\end{theorem}
\begin{proof}
The positive edge-degree ideal $\mathfrak m$ is nilpotent and its quotient is $\kk$, so it is the unique maximal ideal. Looplessness implies $\rho(B)\geq1$ for every nonempty $B$. If $\rho(A)=r$, then $\rho(A\cup B)\leq r<\rho(A)+\rho(B)$ for such $B$. Formula~\eqref{product} therefore kills every product $u_Au_B$, proving one inclusion.

If $\rho(A)<r$, some edge $e\in E\setminus A$ joins different components of $(V,A)$ within a component of $G$. Thus $\rho(A\cup\{e\})=\rho(A)+1$ and $u_Au_{\{e\}}=u_{A\cup\{e\}}\neq0$.

To exclude cancellation, take $v=\sum_A a_Au_A$ with $a_{A_0}\neq0$ for some non-full-rank $A_0$, and choose such an $e$ for $A_0$. Every nonzero summand of $vu_{\{e\}}$ comes from an $A$ not containing $e$. The map $A\mapsto A\cup\{e\}$ is injective on these sets. Hence the coefficient of $u_{A_0\cup\{e\}}$ is exactly $a_{A_0}$ and cannot cancel. This proves the converse for arbitrary linear combinations. Finally $\rho(A)=n-1$ means $k(A)=1$ for connected $G$.
\end{proof}
For disconnected $G$, full rank means that $A$ connects each component of $G$ internally; it does not mean that $(V,A)$ is connected. For the edgeless graph $N(G)=\kk$ and \eqref{soc} remains valid.

\begin{corollary}[Socle Hilbert polynomial]\label{socpoly}
For $c=k(E)$,
\[
 \sum_d\dim_{\kk}(\Soc N(G))_{r,d}s^d
 =\sum_{k(A)=c}s^{|A|}=[t^c]f(G;t,s,1).
\]
Thus the rank/edge-bigraded dimensions of the socle are already determined by the Negami polynomial. All socle elements have rank degree $r$, but may have different edge degrees.
\end{corollary}
The last statement is not an assertion of standard graded levelness: $N(G)$ need not be generated in degree one. Nor does it imply levelness of any broken-circuit ring. For example the triangle socle has three basis elements of edge degree two and one of edge degree three.

\subsection*{Vertex degrees and connected factors}
For each $v\in V$ define $d_v(A)$ as its degree in $(V,A)$, counting parallel edges with multiplicity. Put $\mathbf d(A)=(d_v(A))_{v\in V}$. Every nonzero product in \eqref{product} has disjoint supports, so vertex degrees add. Consequently
\[
 N(G)_{\mathbf a}=\operatorname{span}_{\kk}\{u_A:\mathbf d(A)=\mathbf a\}
\]
defines an $\mathbb N^V$-grading, in addition to rank.
\begin{corollary}[Connected regular spanning subgraphs]\label{factors}
If $G$ is connected and $k\geq0$ is an integer, then
\[
 \boxed{\dim_{\kk}(\Soc N(G))_{(k,\ldots,k)}
 =\#\{A\subseteq E:(V,A)\text{ is connected and }k\text{-regular}\}.}
\]
\end{corollary}
\begin{proof}
Restrict the basis of Theorem~\ref{socle} to vertex degree $(k,\ldots,k)$.
\end{proof}
A $k$-factor here is a spanning $k$-regular subgraph, not a divisor in chip-firing theory. For a simple graph on at least three vertices, the $k=2$ case counts Hamiltonian cycles as edge subsets, without a choice of direction or starting point. This is a counting interpretation, not an efficient algorithm for computing them.

\section{What can the algebra distinguish?}
The comparison with the polynomial requires specifying exactly what structure is retained.
\begin{proposition}[Matroid invariance and marked-edge information]\label{invariance}
An isomorphism of graphic matroids induces a rank/edge-bigraded algebra isomorphism between their $N(G)$ algebras. If the singleton classes $u_{\{e\}}$ are retained as marked elements, their products determine all independent sets of the matroid.
\end{proposition}
\begin{proof}
A matroid isomorphism is an edge bijection preserving ranks. Extend it to subsets and send $u_A$ to the corresponding basis element. Formula~\eqref{product} is preserved. For distinct edges $e_1,\ldots,e_d$, their singleton product is nonzero exactly when the set is independent: every successive rank increment is one precisely in that case. Dependent sets give zero at the first rank failure. The minimal dependent sets then recover the circuits.
\end{proof}
This marked-edge recovery does not prove recovery from the abstract algebra: an arbitrary algebra isomorphism need not preserve the singleton classes. Moreover, graphs with isomorphic cycle matroids cannot be distinguished by the rank/edge-bigraded algebra, even if the graphs are not isomorphic.

Vertex grading is extra incidence information. For example $P_4$ and $K_{1,3}$ both have
\[
 f(G;t,x,y)=t(x+ty)^3,
\]
and their rank/edge-bigraded algebras are isomorphic to $\kk[z_1,z_2,z_3]/(z_1^2,z_2^2,z_3^2)$. Each socle is one-dimensional, but its vertex multidegree is respectively a permutation of $(1,2,2,1)$ and $(3,1,1,1)$. Thus the vertex-graded objects differ even up to permuting vertex coordinates. This elementary example uses the additional vertex grading; it is not a distinction by the underlying algebra. Indeed, vertex degrees of the edge-degree-one classes directly encode endpoint incidence.

\paragraph{Open comparison problem.}
Find two loopless graphs with equal Negami polynomials whose algebras $N(G)$ are nonisomorphic as rank/edge-bigraded algebras, or prove an appropriate determination theorem. Any distinction must go beyond the Hilbert polynomial and the socle dimensions in Corollary~\ref{socpoly}. Possible quantities to examine include the ranks of multiplication maps and the bigraded dimensions of $\mathfrak m^i/\mathfrak m^{i+1}$. A claimed example must include an isomorphism-invariant obstruction, not merely different multiplication tables in chosen subset bases. No such example or determination theorem is established here.

\section{Conclusion and status of the revision}
The classical Hilbert formulas and family computations have been retained in a shorter form. The connectivity assumptions in Theorem 7.1 and Corollary 8.2 are now explicit. The triangle construction has been extended to all finite loopless graphs, with proofs of its multiplication, full Hilbert realization, socle basis and vertex-degree counting interpretation. These results concern an algebra distinct from a broken-circuit Stanley--Reisner ring and do not extend the levelness theorem of \cite{LMM}.

The central remaining research issue is whether multiplication supplies an invariant stronger than the Negami polynomial in a precisely specified category of algebras. The present revision proves the stated elementary results and clarifies that question; it does not claim a resolved novelty assessment or a polynomial-equal pair distinguished by the abstract algebra.


\clearpage
\begin{thebibliography}{12}
\bibitem{Negami} S. Negami, Polynomial invariants of graphs, \emph{Trans. Amer. Math. Soc.} \textbf{299} (1987), 601--622.
\bibitem{OxleyNegami} J. G. Oxley, A note on Negami's polynomial invariants for graphs, \emph{Discrete Math.} \textbf{76} (1989), 279--281.
\bibitem{State} S. Negami and K. Kawagoe, Polynomial invariants of graphs with state models, \emph{Discrete Appl. Math.} \textbf{56} (1995), 323--331.
\bibitem{Whitney} H. Whitney, A logical expansion in mathematics, \emph{Bull. Amer. Math. Soc.} \textbf{38} (1932), 572--579.
\bibitem{BO} T. Brylawski and J. Oxley, The broken-circuit complex: its structure and factorizations, \emph{European J. Combin.} \textbf{2} (1981), 107--121.
\bibitem{BZ} A. Bj\"orner and G. M. Ziegler, Broken circuit complexes: factorizations and generalizations, \emph{J. Combin. Theory Ser. B} \textbf{51} (1991), 96--126.
\bibitem{PS} N. Proudfoot and D. Speyer, A broken circuit ring, \emph{Beitr\"age Algebra Geom.} \textbf{47} (2006), 161--166.
\bibitem{LMM} A. Llamas, J. Mart\'inez-Bernal and C. Merino, On the broken-circuit complex of graphs, \emph{Comm. Algebra} \textbf{38} (2010), 1847--1854. Author manuscript: \url{https://www.matem.unam.mx/~merino/online_papers/pepe.pdf} (consulted October 11, 2026).
\bibitem{Berget} A. Berget, Products of linear forms and Tutte polynomials, \emph{European J. Combin.} \textbf{31} (2010), 1924--1935.
\bibitem{Stanley} R. P. Stanley, \emph{Combinatorics and Commutative Algebra}, 2nd ed., Birkh\"auser, 1996.
\bibitem{Oxley} J. Oxley, \emph{Matroid Theory}, 2nd ed., Oxford University Press, 2011.
\bibitem{LPS} A. Lubotzky, R. Phillips and P. Sarnak, Ramanujan graphs, \emph{Combinatorica} \textbf{8} (1988), 261--277.
\end{thebibliography}
\end{document}